\documentclass[12pt]{amsart}
\usepackage{amssymb,latexsym}
\usepackage{amsthm, amsmath}
\usepackage{amsfonts}
\usepackage{amssymb}
\usepackage{enumerate}

\makeatletter
\@namedef{subjclassname@2010}{%
  \textup{2010} Mathematics Subject Classification}
\makeatother

\newtheorem{thm}{Theorem}[section]
\newtheorem{lem}[thm]{Lemma}
\newtheorem{cor}[thm]{Corollary}

\theoremstyle{definition}

\newtheorem{exa}[thm]{Example}
\newtheorem{rem}[thm]{Remark}

\def\log{\operatorname{log}}

\def\proj{\operatorname{proj}}
\def\dist{\operatorname{dist}}
\def\diam{\operatorname{diam}}
\def\iint{\operatorname{int}}

\newcommand{\C}{{\mathbb{C}}}
\newcommand{\D}{{\mathbb{D}}}
\newcommand{\G}{{\mathbb{G}}}

\newcommand{\R}{{\mathbb{R}}}

\renewcommand{\O}{{\mathcal{O}}}

\numberwithin{equation}{section}

\begin{document}

\baselineskip=17pt

\title[Proper mappings vs. peak points]{Proper holomorphic mappings vs. peak points and Shilov boundary}

\author[\L.~Kosi\'nski]{\L ukasz Kosi\'nski}
\address{Department of Mathematics, Faculty of Mathematics and Computer Science, Jagiellonian University,
\L ojasiewicza 6, 30-348 Krak\'ow, Poland}
\email{lukasz.kosinski@im.uj.edu.pl}

\author[W.~Zwonek]{W\l odzimierz Zwonek}
\address{Department of Mathematics, Faculty of Mathematics and Computer Science, Jagiellonian University,
\L ojasiewicza 6, 30-348 Krak\'ow, Poland}
\email{wlodzimierz.zwonek@im.uj.edu.pl}

\begin{abstract} We present a result on existence of some kind of peak functions for $\C$-convex domains
and for the symmetrized polydisc. Then we apply the latter result to show the equivariance of the set of peak points for $A(D)$ under proper holomorphic mappings. Additionally,
we present a description of the set of peak points in the class of bounded pseudoconvex Reinhardt domains.
\end{abstract}

\thanks{The work is partially supported by the grant of the Polish Ministry for Science and Higher Education No. N N201 361436.}
\keywords{Shilov boundary, peak points, proper holomorphic mappings, $c$-finite compactness, pseudoconvex Reinhardt domains}
\subjclass[2010]{Primary 32T40, Secondary: 32H35, 32A07, 32F45}

\maketitle

\section{Introduction}
The aim of the paper is to present some results on (different kinds of) holomorphic peak points -
for basic definitions and properties on this subject see the survey article \cite{N}.
First we show the existence of a special kind of weak
peak functions for the class of $\C$-convex domains (see Theorem~\ref{th: c-convex}). We also show the existence
of peak functions for the symmetrized polydisc (see Theorem~\ref{prop: pol}). The latter is crucial in the proof of the main result of the paper,
the equivariance of the set of peak points under proper holomorphic mappings (see Theorem~\ref{th:prop}). It is interesting that the symmetrized polydisc, a domain that has been recently extensively studied, is also an important tool in the proof of the invariance of $c$-finite compactness (a notion slightly stronger than
the Carath\'eodory completeness) under proper holomorphic mappings (see Theorem~\ref{th: c-finite}).
Then we give a description of the set of peak points for the class of
pseudoconvex Reinhardt domains (see Theorem~\ref{th: sh}). Finally we deal with the claim of Bremermann on the form of the Shilov boundary
of the schlicht envelope of holomorphy - although we do not know whether Bremermann's claim is correct in general we can prove it
in the class of Reinhardt domains (see Corollary~\ref{cor: envelope}).

\section{Peak functions in the symmetrized polydisc and in $\C$-convex domains}
We start by recalling a description of the symmetrized polydisc which is basic in many proofs that we present in the paper.

Let us recall that the {\it symmetrized polydisc }
is a domain denoted by $\mathbb G_n$ and given by the formula $\mathbb G_n=\pi_n (\mathbb D^n),$
where $\pi_n=(\pi_{n,1},\ldots, \pi_{n,n})$ is defined as follows
$$\pi_{n,k}(\lambda)=\sum_{1\leq j_1<\ldots<j_k\leq n} \lambda_{j_1}\cdots \lambda_{j_k},\quad \lambda\in \mathbb C^n,\ k=1,\ldots,n.$$

$\mathbb D$ denotes the unit disc in $\mathbb C$. We also put $\mathbb T:=\partial\mathbb D$.

Following \cite{Cos} for $n\geq 2$ we define $$p_{n,\lambda}(z)=\tilde z(\lambda)=(\tilde{z}_1(\lambda),\ldots, \tilde z_{n-1}(\lambda))
\in \mathbb C^{n-1},\quad z\in \mathbb C^n,\ \lambda\in \mathbb C, n+\lambda z_1\neq 0,$$ where
$$\tilde z_j(\lambda)= \frac{(n-j) z_j +\lambda (j+1)z_{j+1}}{n+\lambda z_1},\quad 1\leq j\leq n-1.$$
It is known (see \cite{Cos}) that $z\in \mathbb G_n$ if and only if $\tilde z(\lambda)\in \mathbb G_{n-1}$
and $n+\lambda z_1\neq 0$ for any $\lambda\in \overline{\mathbb D}$. Recall that the symmetrized polydisc is a domain that came up naturally a few years ago in the problem of $\mu$-synthesis and turned out to have extremely interesting function geometric properties (see e. g. \cite{AY}, \cite{C2004}, \cite{EZ} and many others).

We shall need the following result which is interesting in its own right.

\begin{thm}\label{prop: pol} For any $a$ in the topological boundary of $\mathbb G_n$ there is a neighborhood
$U$ of the point $a$ and a mapping $\varphi$ holomorphic on $\mathbb G_n\cup U$ such that
$|\varphi|<1$ on $\mathbb G_n$ and $\varphi(a)=1.$
\end{thm}

\begin{proof} We proceed inductively. For $n=1$ the statement is clear.
Let us take $n>1.$ If $|a_1|=n$ then it suffices to define $\varphi(z):=z_1\frac{|a_1|}{na_1}.$
In the other case it follows from the above description of $\mathbb G_n$ that there is a $\lambda\in\overline{\mathbb D}$
such that $p_{n,\lambda}(a)\in \mathbb C^{n-1}\setminus \mathbb G_{n-1}.$ By the continuity argument $p_{n,\lambda}(a)$ lies in $\overline{\mathbb G}_{n-1}$, whence $p_{n,\lambda}(a)$ lies in the topological boundary of $\mathbb G_{n-1}.$
Applying the inductive assumption we find a mapping $\tilde{\varphi}$ peaking at $p_{n,\lambda}(a)$
for $\mathbb G_{n-1}.$ Composing it with $p_{n,\lambda}$ we get a desired mapping.
\end{proof}

In fact it is the above result that we need in the proofs of the main results of the paper. Nevertheless, we present below another result
on existence of a (weaker) peak function in $\C$-convex domains which we find interesting too.
To the authors' knowledge the existence of functions presented in the theorem below is not known.


 First recall that a domain $D\subset\mathbb C^n$ is {\it $\C$-convex} if $l\cap D$
is connected and simply connected for any complex line $l$ intersecting $D$ (for basic properties of $\C$-convex domains see e. g.
\cite{APS}).

We have namely

\begin{thm}\label{th: c-convex} Let $D$ be a $\mathbb C$-convex domain in $\mathbb C^n$.
Then for any $a\in \partial D$ there is a $\varphi\in \mathcal O(D)\cup \mathcal C(D\cup \{a\})$
such that $|\varphi|<1$ on $D$ and $\varphi(a)=1.$
\end{thm}

For the case of a one dimensional bounded domain a stronger result has been obtained in \cite{Be} under the additional assumption that every boundary point of $D$ is linearly accessible from the interior. In the proof below we shall show that the statement of the theorem holds without this assumption. It is possible that the latter result is known. However, we could not find a reference for that result; therefore, we present the proof in this case.

\begin{proof} First we focus on the case when $D$ is a bounded domain of a complex plane.
Let $d=\diam D.$ Since $D$ is simply connected, $f(x)=\log ((x-a)d^{-1}),$ $x\in D,$ is a well defined holomorphic function on $D.$
Then $\varphi=\exp(1/f)$ satisfies the desired properties.

Now we show that \emph{for any domain $D\subset \mathbb C$ and for any $a\in \partial D$ such that
the connected component of $\partial D$ containing the point $a$ contains at least two elements
there is a holomorphic mapping $\varphi$ on $D$ continuous on $D\cup \{a\}$ such that $|\varphi|<1$ on $D$ and $\varphi(a)=1$}.
(Clearly, this statement implies the assertion for an arbitrary simply-connected proper domain of the complex plane).
Actually, let $I$ be the connected component of $\partial D$ containing $a$ and take any nonempty closed connected subset $J$ of $I$
not containing $a$ and such that $J$ is not a point.
Let $C$ be the connected component of $\overline{\mathbb C}\setminus J$ containing $a$.
It follows from the Riemann mapping theorem that there is a conformal mapping $F$ between $C$ and the unit disc.
Applying the former step to $F(\tilde D)$, where $\tilde D$ is a connected component of $\overline{\mathbb C}\setminus I$
containing $D$, we get a function $\tilde{\varphi}$ weakly-peaking at $F(a)$.
Then the composition $\varphi=\tilde{\varphi}\circ F$ satisfies the assertion.

We are left with the case $n>1.$ Consider a complex hyperplane $H$ through $a$ such that $D\cap H=\emptyset$ (see Theorem~2.3.9 in \cite{APS})
and let $l$ be a complex line orthogonal to $H.$ Denote by $\tilde D$ (respectively $\tilde a$) the projection of $D$ (resp. $a$)
onto $l$ in direction $H$. Then $\tilde D$ is a simply-connected domain (see \cite{APS}, Theorem~2.3.6)
and $\tilde a$ is its boundary point. It remains to apply the one dimensional case
and compose the peak function with the projection in direction $H$.
\end{proof}

\begin{rem} The above theorem needs some comment. First, note that the peak functions appearing in 
Theorem~\ref{th: c-convex} are weaker than the ones existing in $\G_n$ (see Theorem~\ref{prop: pol}).
Secondly, one cannot apply the above theorem to $\G_n$ since the symmetrized polydisc $\G_n$ is $\C$-convex iff $n\leq 2$
(see \cite{NPZ}).
\end{rem}

Motivated by the last result we introduce the following definition.
We shall say that a domain $D$ in $\mathbb C^n$ satisfies property ($\dag$)
if for any $a\in \partial D$ there is a $\varphi\in \mathcal O(D,\mathbb D)\cup \mathcal C(D\cup \{a\})$ such that $\varphi(a)=1$. We call such a function a {\it weak peak function at $a$}.

The example of pointed disc shows that being linearly convex is not sufficient for a domain to satisfy the property ($\dag$).

It is clear that any bounded domain satisfying the property ($\dag$) is $c$-finitely compact (for definition see e. g. \cite{JP})
and (hence) Carath\'eodory complete and hyperconvex.
It also follows from Theorem~\ref{prop: pol}
that the symmetrized polydisc $\mathbb G_n$ is $c$-finitely compact which was first proven in \cite{NPZ}.

It seems very probable that the property ($\dag$) is very closely related to the Carath\'eodory completeness of a domain. Recall that any bounded complete Reinhardt domain is $c$-finitely compact (see e. g. \cite{JP}). Therefore, it would be interesting to know whether all bounded complete pseudoconvex Reinhardt domains
satisfy the property ($\dag$).

\section{Peak points under proper holomorphic mappings}
For a bounded domain $D\subset\mathbb C^n$ we denote by $A(D)$ the algebra of all functions holomorphic on $D$
and continuous on $\overline{D}$. The set of peak points for $A(D)$, i. e. such points $z\in\overline{D}$ that there is an $f\in A(D)$
with $|f(z)|>|f(w)|$ for any $w\in \overline{D}\setminus\{z\}$, is denoted by $P(D)$.
Our main result in this section and in the whole paper is the following.

\begin{thm}\label{th:prop} Let $D,G$ be bounded domains in $\C^n$. Let $F: D\to G$ be a proper holomorphic mapping
which extends to a proper holomorphic mapping $U\to V$, $\overline D\subset U,$ $\overline G\subset V$ with the same multiplicity.
Then $a\in \partial D$ is a peak point for $A(D)$  if and only if $F(a)$ is a peak point for $A(G)$.
\end{thm}
The example
$\{z\in\D:\operatorname{Re}z>0\}\owns z\mapsto z^2\in \D\setminus(-1,0]$ shows that some additional assumption 
(except for the continuous extension of $F$ onto $\mathbb{D}$) is needed in Theorem~\ref{th:prop}.

It is known that the Shilov boundary of a bounded domain equals the closure of the set of its peak points (see \cite{G}).
Thus Theorem~\ref{th:prop} has the following interesting corollary. By $\partial_sD$ we denote the Shilov boundary of the bounded domain $D$ in $\mathbb C^n$.

\begin{cor} Let $F$ be as above. Then $F^{-1}(\partial_s G)=\partial_s D.$
\end{cor}

The above result generalizes a result on equivariance of the Shilov boundary under proper
holomorphic mappings from \cite{K}.

Theorem~\ref{th:prop} is a consequence of two following more general lemmas.

\begin{lem}\label{lem:nzm} Let $D,G$ be bounded domains in $\C^n$, $a\in\partial D$.
Let $F:D\mapsto G$ be a proper holomorphic mapping of multiplicity $m$
extending continuously onto $\bar D$ and such that the following property is satisfied.

For any $b=F(a)\in\partial G,$ there is a set $\{a_1,\ldots,a_l\}\subset\partial D$ of pairwise different points
and a sequence of positive integers $n_1,\ldots,n_l$ such that for any sequence $G\owns b_k\to b$
such that $F^{-1}(b_k)=\{a_{k,1},\ldots,a_{k,m}\}$ after a permutation (if necessary) the sequences $(a_{k,j})_j$
are convergent to elements $a_1,\ldots,a_l$ and for any $k=1,\ldots,l$
the number of sequences $(a_{k,j})_j$ tending to $a_k$ is equal to $n_k$.

If $a$ is a peak point for $A(D)$, then $F(a)$ is a peak point for $A(G)$.
\end{lem}

\begin{proof} Let $\{a_1,\ldots,a_l\}$ be as above where $a_1=a$. Let $\phi\in A(D)$ be such that $\phi(a)=1$,
$|\phi(z)|<1$ for any $z\in \bar D\setminus\{a\}$. Consider the mapping $g:=\pi_m\circ (\phi\times\ldots\times\phi)\circ F^{-1}$.
Then $g\in\O(G,\G_m)$ and because of properties imposed on $F$ we get that $g$ extends continuously
to $\bar G$, $g(\bar G\setminus\{b\})\subset\G_m$ and $g(b)\in\partial\G_m$.
Now let $\varphi:\G_m\cup\{g(b)\}\mapsto\D\cup\{1\}$ be a continuous function,
holomorphic on $\mathbb G_m$ such that $\varphi(g(b))=1$ and $|\varphi(z)|<1$, $z\in\G_m$
(such a function exists by Theorem~\ref{prop: pol}). Then $\varphi\circ g$ is an $A(G)$-peak function for $b$.
\end{proof}

\begin{rem}
 Using similar reasoning one may show that $F$ maps peak points for $A^{k}(D)$
($k\in \mathbb N\cup\{\infty,\omega\}$) to peak points for $A^k(D)$ provided that $F$ extends to a proper and holomorphic
mapping $U\mapsto V$ where $U\supset\bar D$, $V\supset\bar G$ and $D=f^{-1}(G)$.
Actually, it suffices to make use of the fact that Theorem~\ref{prop: pol} guarantees that the mapping $\varphi \circ g$ occurring
in the proof of the above lemma is of class $A^k$ provided that $\phi$ is.
\end{rem}

\begin{lem}\label{lem: ea} Let $D,G$ be bounded domains in $\C^n$.
Let $F:D\to G$ be a proper holomorphic mapping extending continuously to $\bar D$
and such that the fibers $F^{-1}(y)$ are finite for any $y\in \bar G.$ Then if $y$ is a peak point for $A(G)$
then any point of $F^{-1}(y)$ is a peak point for $A(D)$.
\end{lem}
\begin{proof} The properness of $F$ implies that the fibers $F^{-1}(y)$ are non-empty.
Fix $y\in P(G)$ and let $\varphi\in A(G)$ be a function peaking at $y.$
Put $E:=F^{-1}(y)=\{x^1,\ldots,x^k\}.$ It is clear that $E$ is a peak set and $\psi=\varphi\circ F$ peaks at $E.$
Take any $j=1,\ldots,k.$ For any $l\neq j$ chose $\sigma_l=1,\ldots,n$ such that $x^{j}_{\sigma_l}\neq x^l_{\sigma_l}.$
Let us define $\tilde{\Phi}(\lambda) =  \prod_{l\neq j} (\lambda_{\sigma_l} - x^l_{\sigma_l})$
and $\Phi(\lambda)=\tilde{\Phi}(\lambda)\tilde{\Phi}(x^j)^{-1}.$
Now it suffices to apply $1/4-3/4$ method of Bishop (see e.g. \cite{N})
to the functions of the form $\psi^{l}(x) \Phi^{m}(x)$ with suitably chosen $l,m\in \mathbb N.$
\end{proof}

\begin{rem} Note that Lemma~\ref{lem: ea} remains true if $F:D\to G$ is a holomorphic mapping extending continuously
to a surjective mapping $\bar D\to \bar G$ whose fibers $F^{-1}(y)$ are finite whenever $y\in \partial G.$

What is more, since any peak set contains a peak point (see \cite{G})
we get that $P(G)\subset F(P(D))$ (where $P(D)$ denotes the set of peak points with respect to $A(D)$)
for any holomorphic mapping $F:D\to G$ which extends continuously to a surjective mapping $\bar D\to \bar G$.
\end{rem}

\begin{rem} We do not know whether for a proper holomorphic mapping $F:U\supset \bar D\to \bar G\subset V$
the inclusion $F^{-1}(A^k(G))\subset A^k(D)$ is true. If for example $D$
is smooth, then the $\bar{\partial}$-problem has a smooth solution on $\bar D$ (Kohn's Theorem (see \cite{Ko})).
Thus one may repeat the argument of Pflug from \cite{P} (see also \cite{N}) to get an affirmative answer to this problem.
\end{rem}

\section{Invariance of $c$-finite compactness under proper holomorphic mappings}
Now we show another result on invariance under proper holomorphic mappings. It turns out that
$c$-finite compactness is invariant with respect to proper holomorphic mappings (for definition of Carath\'eodory
(pseudo)-distance and a $c$-finitely compact domain see e.g. \cite{JP}) and in the proof of the result below we use once
more the properties of the symmetrized polydisc. The result below is motivated by the invariance of 
hyperconvexity and pseudoconvexity under proper holomorphic
mappings.

\begin{thm}\label{th: c-finite}
For any bounded domains $D,G\subset\C^n$ and for any proper, holomorphic mapping $F:D\mapsto G$ $c$-finite compactness of $D$
is equivalent to the $c$-finite compactness of $G$.
\end{thm}

\begin{proof} The sufficiency of $c$-finite compactness of  $G$ for $c$-finite compactness of $D$
follows from the holomorphic contractibility of $c$ and properness of $F$. To show sufficiency of $c$-finite compactness of $D$
for $c$-finite compactness of $G$ we proceed as follows.

Fix $w_0=F(z_0)\in G$. For any $w\in G$ we take a point $z$ from $F^{-1}(w)$
and choose a function $f\in\O(D,\D)$ with $f(z_0)=0$ and $|f(z)|=c_D^*(z_0,z)$.
Then the function $g_w:=\pi_m\circ(f\times\ldots\times f)\circ F^{-1}$ is a well-defined holomorphic function on $G$,
$g_w(G)\subset\G_m$ such that $g_w(w_0)$ lies in some fixed compact $K\subset\G_m$ and $g_w(w)\to\partial\G_m$ as $w\to\partial G$
(here we use the fact that $F$ is proper and $D$ is $c$-finitely compact).
Since $c_{G}(w_0,w)\ge c_{\G_m}(g_w(w_0),g_w(w))$ the $c$-finite compactness of $\G_m$ gives that $c_G(w_0,w)\to\infty$ which
finishes the proof.
\end{proof}

\begin{rem}
Recall here that $c$-finite compactness implies Carath\'eodory ($c$-)completeness.
Moreover, the question whether $c$-completeness implies $c$-finite compactness is not solved yet (see e. g. \cite{JP}).
It would be interesting to know if the notion of $c$-completeness is invariant under proper holomorphic mappings.
It is also interesting to know whether notions of completeness with respect to other invariant functions
(e.g. Kobayashi, Bergman or inner Carath\'eodory) remain invariant under proper holomorphic mappings.
In this context recall
that it is well-known that the Kobayashi completeness is invariant under holomorphic coverings (see e. g. \cite{JP}).
\end{rem}

\section{Peak points of Reinhardt domains}

Below we denote $M:=\{z\in\mathbb C^n: z_1\cdot\ldots\cdot z_n=0\}$. Recall also that a point $x$ of a convex set $X\subset\mathbb R^n$
is called {\it extremal} if there are no two different points $y,z\in X$ such that $x=\frac{y+z}{2}$.

In \cite{G1} a description of the set of peak points of compact, Reinhardt sets with respect to $H(K)$ was given,
where $H(K)$ is the uniform closure of the space of functions holomorphic in a neighborhood of $K$, was given.

It this section we shall describe the set of peak points of bounded Reinhardt domains of holomorphy with respect to $A(D)$.
In Section~6 this result will be extended to arbitrary bounded Reinhardt domains. We introduce some notation. Namely for a point
$z=(z_1,\ldots,z_n)\in D\subset\C^n$ with $z_j\neq 0$ we denote $\log z:=(\log|z_1|,\ldots,\log|z_n|)\in\R^n$ and $\exp(z):=(e^{z_1},\ldots, e^{z_n})$. For a Reinhardt domain $D\subset\C^n$
we denote $\log D:=\{\log z:z\in D, z_j\neq 0, j=1,\ldots,n\}$. Moreover, for a set $A\in \mathbb R^n$ we denote $\exp(A):=\{z\in \mathbb C^n:\ \log|z|\in A\}.$
For $z=(z_1,\ldots, z_n),$ $w=(w_1,\ldots, w_n)\in \mathbb C^n,$ we put $z\cdot w=(z_1w_1,\ldots, z_nw_n).$ 

\begin{lem}\label{lem: 1} Let $D$ be a bounded pseudoconvex Reinhardt domain in $\mathbb C^n$.
Let $z_0\in \partial D\setminus M$ be such that $\log z_0$ is an extremal point of $\overline{\log D}.$
Then there is a sequence $(f_\mu)\subset A(D)$ of Laurent monomials such that $f_{\mu}(z_0)\to 1$,
$||f_{\mu}||_D\to 1$ and $\lim_{\mu\to \infty} ||f_{\mu}||_{D\setminus U} =0$ for any Reinhardt neighborhood $U$ of $z_0.$
\end{lem}

\begin{proof} We proceed inductively. For $n=1$ the assertion is trivial. Let $n>1.$ Fix $\epsilon>0,$ $N>0$
and $U$ a Reinhardt neighborhood of $z_0$. We are looking for a Laurent monomial $f$ in $A(D)$
such that $||f||_D<e^{\epsilon},$ $||f||_{D\setminus U}<e^{-N}$ and $f(z_0)>e^{-\epsilon}.$ Put $x_0:=\log z_0$, $G:=\log D$,
$V=\log U$ and let $\{L=l_0\}$, $Lx=\langle l,x\rangle ,$ $x\in \mathbb R^n,$ $l=(l_1,\ldots, l_n)\in \mathbb R^n$, $l\neq 0,$
be a supporting hyperplane to $G$ at $x_0,$ $G\subset \{L<l_0\}.$
We lose no generality assuming that $l_n\neq 0$ and that $V$ is bounded. Dividing $l$ and $l_0$ by $l_n$ we may assume that $l_n=1.$ Clearly $\bar G\cap \{L=l_0\}$ is convex
and $x_0$ is its extremal point. Let $G'=\proj_{n-1}(\bar G\cap \{L=l_0\})$ be the projection of $\bar G\cap \{L=l_0\}$
onto the first $n-1$ variables.  $G'$ is a convex subset and $x_0$ is its extreme point.
Note that its interior may be empty (viewed as a subset of $\mathbb R^{n-1}=\R^{n-1}\times \{0\} $).
However we may increase it, if necessary, to the closure of a convex domain such that $x_0'=\proj_{n-1}(x_0)$
is its extreme point and $\iint(\exp(G'))$ is a bounded Reinhardt domain of $\mathbb C^{n-1}.$ To increase it we may proceed as follows. Take a supporting hyperplane at $x_0'$ to $G'$.
This hyperplane divides $\R ^{n-1}$ into two sets. Choose the one containing $G'$
and take from its interior any $n-1$ vectors $y_1,\ldots, y_{n-1}$ such that $y_1-\tilde y,\ldots, y_{n-1}-\tilde y$ are linearly independent for some $\tilde y\in G'$. Direct calculations show that the convex hull of $G'\cup \{y_1,\ldots, y_{n-1}\}$ satisfies the desired property.

This allows us to apply the inductive assumption to the Reinhardt domain $\iint(\exp (G'))$.
Then, passing to the logarithmic image for any $\epsilon>0$ we find a linear mapping $T$,
$Tx'=\langle t',x'\rangle ,$ $x'\in \mathbb R^{n-1},$ $t'\in \mathbb Z^{n-1}$
and $t_0\in \mathbb R$ such that $T<t_0+\epsilon/2$ on $G',$ $|Tx_0'-t_0|<\epsilon/2$
and $T<-2N$ on $G'\setminus V'$, where $V'=\proj_{n-1}(V\cap \{L=l_0\})$.

Observe that for $M$ big enough the mapping $$T_M:x\mapsto M(Lx-l_0)+Tx'-t_0$$ is bounded from above on $G$ by $\epsilon$
and less than $-N$ on $G\setminus V.$

Actually, seeking a contradiction assume that it is not true. Then there exists a sequence $(M_n)\subset \mathbb R_{>0}$
converging to $\infty$ such that one of the two possibilities holds: either there is a sequence $(y_n)\subset G$
converging to $y_0$ such that $T_{M_n}(y_n)\geq\epsilon$ if $y_0\in V$ and $T_{M_n}(y_n)>-N$
if $y_0\not\in V$ or there is a sequence $(y_n)$ with no limit point in $\bar G$ such that $T_{M_n}(y_n)>-N.$
If the first possibility holds we immediately find that $L(y_0)=l_0,$ so in this case a contradiction follows from
the properties of the mapping $T.$

In the second case one may easily find a sequence of positive numbers $(p_n)$ such that $(p_ny_n)$
converges to $\tilde y\in (\mathbb R^n)_*$. Let $t>0.$ Clearly, the set $\bar G\cap \{T_{M_n}>-N\}$
is convex and contains $tp_n y_n + (1-tp_n)x_0$ for sufficiently large values of $n.$
Making use of the inequalities $T_{M_n}(tp_n y_n + (1-tp_n)x_0)>-N$ and $L(tp_n y_n + (1-tp_n)x_0)\leq l_0$
we easily find that $L(t\tilde y+ x_0)=l_0.$ What is more, the inequality $T_{M_n}(tp_n y_n + (1-tp_n)x_0)>-N$
implies that $T(t\tilde y'+x_0')-t_0\geq-N$. Collecting the above mentioned facts and making use of the behavior of $T$
on the hyperplane $\{L=l_0\}$ we get a contradiction for $t>0$ big enough (i.e. for $t$ such that $x_0+t\tilde y\not\in V$).

In the same way one may show the existence of $\delta >0$ such that $x\mapsto T_M x+\langle \alpha,x\rangle$
has the same properties provided that $||\alpha||<\delta$.

It follows from the multidimensional Dirichlet Theorem that for any $\mu\in \mathbb N$
there are integers $\alpha_{\mu,1},\ldots,\alpha_{\mu,n}$ and an integer $k_{\mu}$
such that $|l_j-\frac{\alpha_{\mu,j}}{k_{\mu} }|\leq \frac{1}{\mu k_{\mu}},$ $j=1,\ldots,n.$
Losing no generality we may assume that $(k_{\mu})_{\mu}$ converges to $\infty$.

Finally it suffices to observe that the functions defined by the formula
$$F_{\mu}(z)=e^{i\theta_{\mu} -l_0 k_{\mu}-t_0} z^{t+\alpha_{\mu}},$$
where $t=(t',0)\in \C^n$, satisfy the desired property for some $\theta_{\mu}\in \mathbb R$, $\mu>>1.$
Actually, let us compute $$\log F_{\mu}(e^x)=k_{\mu}(Lx-l_0)+ Tx-t_0+ \sum_{j=1}^n x_j(\alpha_{\mu,j}-k_{\mu}l_j).$$
Thus it is enough to make use of the properties of the constructed functions.
\end{proof}

\begin{thm}\label{th: sh} Let $D$ be a bounded pseudoconvex Reinhardt domain in $\mathbb C^n$.
Then $z\not\in M$ is a peak point with respect to $A(D)$ iff $\log z$ is an extremal point of the closure of the logarithmic image of $D.$

Moreover, $0$ is not a peak point for $A(D)$ and $z=(z',0)\in \mathbb C^{n-1}\times \{0\}$ is a peak point for $A(D)$
if and only if $z'$ is a peak point of $\proj_{n-1} (D)$
and $(\{z'\}\times \mathbb C)\cap \bar D = \{z\}.$
\end{thm}

Note that Lemma \ref{lem: 1} provides a tool which allows us to use the Bishop's method of constructing peak functions.
Below we present a slightly different proof relying upon the methods developed by Gamelin and involving a concept of representing
measures (see \cite{G}).

\begin{proof} Some ideas of the proof are derived from \cite{G1}. It is clear that any $z\in \partial D\setminus M$ such that $\log z$
is not an extremal point of $\overline{\log D}$ lies in an analytic disc contained in $\overline D$, hence it is not a peak point.

We shall show that every $z\in \partial D\setminus M$ such that $\log z$ is extremal in $\overline{\log D}$ is a peak point for $A(D).$
Let $\nu$ be a representing measure for $A(D)$. Applying the previous lemma to the point $z$ we get a sequence $(f_{\mu})$ of Laurent monomials with properties as in Lemma~\ref{lem: 1}. In particular, $f_{\mu}(z_0)=\int_D f_{\mu}d\nu$ and $f_{\mu}$ is a bounded sequence in $L^{\infty}(d\nu).$ Let $f$ be a weak-star limit point of $f_{\mu}.$  It follows from the choice of $f_{\mu}$ that $||f||_D \leq 1$ and $1=|\int_D f d\nu|.$ Thus $|f|=1$ $\nu$-almost everywhere and what follows the support of the measure $\nu$ is contained in the torus $z_0\cdot \mathbb T^n.$

Applying the formula $f(z_0)=\int f d\nu$ to every component of $g(z)=e^{i\theta}\cdot z,$ where $\theta \in \mathbb R^n$ is such that $|z_0|=e^{i\theta}\cdot z_0$, we easily find that $\mu$ is a Dirac measure supported at $z_0.$

Now suppose that $0\in \overline D.$ Then there are $\alpha_1,\ldots,\alpha_n\in \mathbb R_{\leq 0}$ not vanishing simultaneously such that $x_0+E(\alpha):=x_0+\{(\alpha_1 t,\ldots,\alpha_n t):\ t>0\}\subset \log D$ for some (so by a simple convexity argument for any) $x_0\in \log D.$ To simplify the notation assume that $x_0=0.$ Seeking a contradiction suppose that $f\in A(D)$ peaks at $0.$ Observe that $v(\lambda)=(\sup |f(\lambda^{\alpha})|)^*$, where the supremum is taken over the branches of $\lambda^{\alpha}=(\lambda^{\alpha_1}, \ldots, \lambda^{\alpha_n})$, is a subharmonic function on the unit disc of the complex plane attaining its maximum at $0$. Thus $v$ is constant. From this we easily get a contradiction.

The rest part of the theorem is self-evident.
\end{proof}

\section{Some remarks on Bremermann's claim}
Bremermann claimed in \cite{B} that the Shilov boundary of any domain with the schlicht envelope of holomorphy
coincides with the Shilov boundary of its envelope of holomorphy.
Note that $A(\widehat D)\subset A(D)$ so we immediately see that $\partial_s\widehat D\subset \partial_s D$.
Bremerman stated that the inclusion $\partial_s D\subset \partial_s \widehat D$ is trivial, however we do not know whether
this statement is true. That would follow from the inclusion $A(D)\subset A(\widehat D)$ which however does not hold in general
as the following example (in the class of bounded Reinhardt domains) shows

\begin{exa} Consider the Reinhardt domain $D\subset\mathbb C_*\times \mathbb C$ such that $\mathbb D_*\times\{0\}\subset D$
and \begin{multline*} \log D= \{(x,y) \in \mathbb R_{<1}\times \mathbb R_{<0}:\ y<0 \ \text{if}\ x\in (0,1) \
\text{and}\\ y<-n^2-n\ \text{if}\ x\in (-n^2,-(n-1)^2],\ n\in \mathbb N\}.
 \end{multline*} Observe that $\widehat D=\{(z,w)\in \mathbb C^2:\ |w|<|z|<1\ \text{or}\ 1\leq |z|<e,\ |w|<1\}.$
Moreover, one may check that the function $f(z,w):=\frac{w}{z}$ is an element of the algebra $A(D)$
obviously not belonging to $A(\widehat D).$
\end{exa}

Although we do not know whether the claim of Bremermann holds in general we know that
in the class of Reinhardt domains the claim of Bremermann does hold.

Recall that if $D$ is a Reinhardt domain, then the (schlicht) envelope of holomorphy $\widehat D$ exists and
$\widehat D$ is a Reinhardt domain.

\begin{cor}\label{cor: envelope} Let $D$ be a bounded Reinhardt domain in $\mathbb C^n.$
Then its set of peak points coincides with the set of peak points of the envelope of holomorphy of $D.$
\end{cor}

\begin{proof}  It suffices to show that if $z_0\in \partial D$ is a peak point for $A(D)$, then it is also a peak point for $A(\widehat D).$ Put $G=\log D.$

First suppose that $z_0$ omits $M$. We aim at showing
that $\log z_0$ is extremal in $\overline{conv(G)}$ (here $conv(G)$ stands for the convex hull of $G$,
i.e. $conv(G)=\log\widehat D$). To get a contradiction assume that $\log z_0$ is not an extreme point of $\overline{conv(G)}$.

We shall make use of a few definitions. Following \cite{Z} for a convex domain $U$ of $\mathbb R^n$ we define $$\mathfrak C=\mathfrak C (U)=\{\alpha\in (\mathbb R^n)_*:\ x+\alpha\mathbb R_+\subset U\ \text{for some (any)}\ x\}.$$
Boundedness of $D$ implies that $\mathfrak C\subset (\mathbb R_{\leq 0})^n.$ Note that for any $x_0\in \overline U$ there are $m\geq 1,$ $x_1,\ldots,x_m$ extremal in $\overline U$, $p_1,\ldots,p_m>0$,
$\sum_{j=1}^m p_j=1,$ $\alpha\in \mathfrak C$ and $t_0\geq 0$ such that
\begin{equation}\label{wypfor} x_0=\sum_{j=1}^m p_jx_j+\alpha t_0.\end{equation}
To prove it we proceed inductively. Let $n>1$ (for $n=1$ there is nothing to prove).
The statement follows immediately from the Krein-Milman theorem provided that $U$ is bounded so suppose that $U$ is unbounded
and take any $\alpha\in \mathfrak C$ (unboundedness of $U$ implies that $\mathfrak C(U)\neq \emptyset$).
Then there is $t_0\geq 0$ such that $y_0=x_0-\alpha t_0\in \partial U.$
Let $H$ be a supporting hyperplane to $U$ at $y_0.$
Applying the inductive assumption to $y_0$ and $H\cap \overline{U}$
(increase it to a convex domain of $\mathbb R^{n-1}$ if necessary) and using the fact that $\mathfrak C$
is a cone we easily obtain the formula \eqref{wypfor}.

Now we come back to the proof of the corollary. Let $f\in A(D)$ be a peak function at $z_0$
and denote by $\hat f$ its extension to $\widehat D.$
Applying a decomposition \eqref{wypfor} to $x_0=\log z_0$ and $U=conv(G)$ we obtain $x_1,\ldots,x_m$ extremal in $\overline{conv(G)}$
and $t_0\geq 0$ fulfilling \eqref{wypfor}. Every point $z_j\in \mathbb C^n$ whose logarithmic image equals $x_j$ lies in the Shilov boundary of $\widehat D$, $j=1,\ldots,m$. Since $\partial_s \widehat D\subset \partial_s D$ we find that $x_1,\ldots,x_m\in\overline G$.

First consider the case $t_0=0.$ Since $\log z_0$ is not extremal in $\overline{conv(G)}$, $\log z_0\neq x_j,$ $j=1,\ldots,m$.
Take any $y_j\in G$ close to $x_j$, $j=1,\ldots,m,$ and consider a closed simplex $\mathfrak S$
with vertices $y_1,\ldots,y_m$. Clearly $\hat f$ is holomorphic in a neighborhood of $\exp(\mathfrak S)$.
Let $\mathfrak S_{\epsilon}=\{x\in \mathbb R^n: \dist(x, \mathfrak S)<\epsilon\}$ be an $\epsilon$-hull of $\mathfrak S$. Then $S_{\epsilon}:=\exp(\mathfrak S_{\epsilon})$ is a Reinhardt domain and its Shilov boundary is close to $\{x_1,\ldots, x_m\}$ (in the sense of the Hausdorff distance) provided that $\epsilon$ is close to $0$ and $y_j$ are close to $x_j,$ $j=1,\ldots, m.$ Using Theorem~\ref{th: sh} applied to $S_{\epsilon}$ one can find that $|f|$ attains values close to $1$ on an arbitrary small neighborhood of $\{z\in \mathbb C^n:\ \log |z|=y_j\ \text{for some}\ j\}$.
This gives an immediate contradiction with the fact that $f$ peaks at $z_0$ (as $|f|$ is small far away of the point $z_0$).

If $t_0>0$ then again we take $y_j\in G$ sufficiently close to $x_j$. Put $w_0:=\exp(\sum_{j=1}^m p_j y_j)\in G.$ Note that $w_0 \cdot \lambda^{\alpha}\in \widehat D$ for any $\lambda\in \mathbb D_*$. Therefore $u(\lambda)=(\sup |\hat f(w_0 \cdot \lambda^{\alpha})|)^*,$ where the supremum is taken over all branches of $\lambda^{\alpha}=(\lambda^{\alpha_1},\ldots, \lambda^{\alpha_n})$, is subharmonic on a neighborhood of $\overline{\mathbb D}$ (it may be extended through $0$ because it is bounded). By the maximum principle applied to the subharmonic function $u$ we get that $\max_{e^{-t_0}\partial \D} u\leq \max_{\partial \D}u.$ Thus, similarly as in the case $t_0=0$ we find that there is $z\in \bar D$ close to $e^{y_j}$ for some $j=1,\ldots,m,$ such that $f(z)$ is close to $1$. This is contrary to the fact that $f$ peaks at $z_0$.

In the case when $z_0\in M$ we proceed similarly. Actually, one may easily see that $0\not\in \partial_s D$. Suppose that $z_0\in M\setminus \{0\}$. Losing no generality assume that $z_0=(z_0',0)\in \mathbb C^k_*\times \mathbb C^{n-k}$. Note that $\bar D\cap (\{z_0'\}\times \mathbb C^{n-k})=(z_0',0)$ (use the fact that $z_0$ peaks for $A(D)$). If $z_0'$ is not a peak point of the projection of $\widehat D$ onto $\C^k$, then to get a contradiction one may use the reasoning from the first part of the proof.
\end{proof}

\begin{cor}Let $D$ be a bounded Reinhardt domain in $\mathbb C^n.$ Then its Shilov boundary coincides with the Shilov boundary of its envelope of holomorphy.
\end{cor}

\end{document}